\documentclass[journal]{IEEEtran}
\RequirePackage[OT1]{fontenc}
\RequirePackage{amsthm,amsmath,amsfonts,amssymb}
\RequirePackage[numbers]{natbib}
\RequirePackage[colorlinks,citecolor=blue,urlcolor=blue]{hyperref}
\usepackage{lineno,xcolor}

\newcommand{\bs}{\mu}

\newcommand{\bS}{S}

\newcommand{\tr}{\mathrm{tr}}

\newcommand{\bu}{u}

\newcommand{\um}{\breve{m}}
\newcommand{\toas}{\overset{\text{a.s.}}{\to}}

\newcommand{\toprob}{ \overset{\mathrm{p}}{\to} }

\newcommand{\fH}{\mathcal{H}}

\newcommand{\bfy}{\mu}
\newcommand{\bfA}{A}
\newcommand{\rh}{\hat{R}}

\newcommand{\td}{\tilde{d}}
\newcommand{\coloneqq}{:=}

\newcommand{\rt}{\tilde{R}}

\newcommand{\cplus}{\mathbb{C}^+}
\newcommand{\dt}{\tilde{\delta}}
\newcommand{\dnjstar}{d^*_{nj}}

\newcommand{\bbF}{\mathbb{C}}
\newcommand{\mypr}{\Pr}

\newtheorem{lem}{Lemma} 
\newtheorem{thm}{Theorem}
\newtheorem{prop}{Proposition}

\theoremstyle{definition}
\newtheorem{definition}{Definition}

\theoremstyle{remark}
\newtheorem{rem}{Remark}

\ifCLASSINFOpdf
\else
\fi
\begin{document}
%
\title{High-Dimensional Covariance Shrinkage for Signal Detection}  
%
%
%

\author{Benjamin~D.~Robinson,~\IEEEmembership{Member,~IEEE}, 
        Robert~Malinas,~\IEEEmembership{Student Member,~IEEE}, Alfred O. Hero III,~\IEEEmembership{Fellow,~IEEE}
\thanks{This work was generously supported by AFOSR grants 19COR1936 and 22RYCOR007, and ARO grant W911NF-15-1-0479}
\thanks{Benjamin Robinson is with the US Air Force Research Lab.}
\thanks{Robert Malinas and Alfred Hero are with University of Michigan.}
\thanks{Manuscript received November 19, 2021}}

\maketitle

\begin{abstract}
In this paper, we consider the  problem of determining the presence of a given signal in a high-dimensional observation with unknown covariance matrix by using an adaptive matched filter.  Traditionally such filters are formed from the sample covariance matrix of some given training data, but, as is well-known, the performance of such filters is poor when the number of training data $n$ is not much larger than the data dimension $p$.  We thus seek a covariance estimator to replace sample covariance.  To account for the fact that $n$ and $p$ may be of comparable size, we adopt the ``large-dimensional asymptotic model'' in which $n$ and $p$ go to infinity in a fixed ratio.  
Under this assumption, we identify a covariance estimator that is asymptotically detection-theoretic optimal within a general shrinkage class inspired by C. Stein, and we give consistent estimates for conditional false-alarm and detection rate of the corresponding adaptive matched filter.
\end{abstract}


%
\IEEEpeerreviewmaketitle

\section{Introduction}

In this paper we consider the problem of determining signal presence in a high-dimensional observation with unknown covariance matrix.
To take advantage of the observation's high-dimensionality, we embed the problem
in a sequence of hypothesis testing problems indexed by $n$ that are of non-decreasing
dimension $p_{n}$. We assume that the signal $\mu_n\in \bbF^{p_n\times 1}$ is known modulo scalar multiplication
and that the test observation $y_{n}\in\bbF^{p_{n}\times 1}$ is given, 
but
that the only other data available in the $n^{\text{th}}$
problem are $n$ training data $x_{1},x_{2},\dots , x_{n}\in\bbF^{p_{n}\times 1}$ such that $x_1,x_2,\dots,x_n,(y_n-\mathbb{E}y_n)$ are i.i.d. 
 This formulation is relevant to high-dimensional detection problems that arise in radar/sonar beamforming \cite{scharf1994matched}, climatology \cite{ribes2009adaptation}, and hyperspectral imaging \cite{theilera2007beyond}, among other places. 

The simplest possibility is that $p_{n}$ is constant, $\mu_n=\mu$ is constant, and $y_n=y$ is independent of $n$ and distributed as a multivariate Gaussian with Hermitian positive-definite covariance matrix $R$. In this case, to goal is to determine whether $y$ has mean $a\mu$ for $a=0$ (signal absent) or $a\ne 0$ (signal present).  
More precisely, we wish to decide between the hypotheses
\[
\begin{cases}
\fH_0 : & y =d \\
\fH_1 : & y= a\mu + d, \qquad (a\ne 0)
\end{cases}
\]
where $d$ is a mean-zero Gaussian random vector with covariance matrix $R$.
If $R$ is known, a popular test is the Generalized Likelihood Ratio (GLR) test \cite[Equation~7]{robey1992cfar}: $|T(\mu, R, y)|^2 \gtrless t$, where 
\begin{equation} \label{eq:pop-discrim}
T(\mu, R, y)=\mu'R^{-1}y / (\mu' R^{-1} \mu)^{1/2}
\end{equation}
and $(\cdot)'$ denotes the (conjugate) transpose.
Suppose more realistically that the covariance matrix $R$ 
is unknown but $n\gg p_n$. In this case, a popular replacement for the GLR test is the traditional Adaptive Matched Filter (AMF) test \cite[Equation~8]{robey1992cfar}: $|T(\mu, S_n, y)|^2 \gtrless t$, where
$S_n$ is the $p_{n}\times p_{n}$ sample covariance matrix of the training data:
\[
S_{n}=n^{-1}\sum_{j=1}^{n}x_{j}x_{j}'.
\]
This detector is asymptotically optimal in the sense that it attains the power of the GLR test in the limit as $n\to \infty$. In other words, for any fixed conditional false-alarm rate $\mypr[|T(\mu,S_n,y)|^2 > t \mid \fH_0, x_1, x_2, \dots, x_n]$, the conditional detection rate $\mypr[|T(\mu,S_n,y)|^2 > t \mid \fH_1, x_1, x_2, \dots, x_n]$ converges almost surely to the detection rate of the GLR test. (This follows, for example, from \cite[Section~2]{robinson2020space} and the law of large numbers applied to $S_n$.)

In the modern statistical context, though, the assumption  that $n\gg p_n$ is often not realistic \cite{johnstone2001distribution}. Instead, we follow many authors in assuming the large-dimensional asymptotic limit, in which $n,p_n\to\infty$ and $p_n/n$ converges to a fixed constant $\gamma\in(0,\infty)$ \cite{zhao2006model, donoho2006most, fan2008sure}.
This assumption permits any nontrivial limiting ratio of $p_n$ and $n$---even ones that correspond to the ``sample-starved'' case ($n < p_n$), which is important in applications such as linear discriminant analysis \cite{hastie1995penalized}, principal component analysis \cite{johnstone2001distribution},  space-time adaptive processing \cite{ward1998space}, and other applications such as financial time-series analysis \cite{ledoit2017nonlinear}.  
In this context, we reformulate the hypothesis test above as a sequence of hypothesis tests of increasing dimension:
\begin{equation} \label{eq:main-hyps}
\begin{cases}
\fH_0^n : & y_n = d_n \\
\fH_1^n : & y_n = a\mu_n + d_n, \qquad (a\ne 0),
\end{cases}
\end{equation}
where $d_n$ has $p_n\times p_n$ covariance matrix $R_n$.
The goal is, then, to asymptotically optimize and analyze the ``plug-in'' AMF test $|T(\mu_n, \rh_n, y_n)|^2 \gtrless t$, where $\rh_n$ is some estimator of $R_n$. More precisely, we would like to asymptotically analyze the conditional probability of detection
\begin{align} 
& p_1^n(\hat{R}_n, t) \nonumber \\ & := \mypr[|T(\mu_n, \hat{R}_n, y_n)|^2 >t \mid \fH_1^n, \mu_n, x_1,x_2,\dots, x_n] \nonumber 
\end{align}
and the conditional probability of false-alarm 
\begin{align} 
& p_0^n(\hat{R}_n, t) \nonumber \\ & := \mypr[|T(\mu_n, \hat{R}_n, y_n)|^2 >t \mid \fH_0^n, \mu_n, x_1, x_2, \dots, x_n], \nonumber 
\end{align}
and asymptotically maximize the former subject to a constraint on the latter.
(Here, we condition on $\mu_n$ because we allow the possibility that $\mu_n$ may be governed by a probability distribution.) It turns out that simply choosing $\rh_n = S_n$ can be highly sub-optimal due to the dispersed nature of the eigenvalues of $S_n$.

Stein  \cite{stein1975estimation,stein1986lectures} suggested that for many applications a significant improvement over $S_n$ can be obtained by ``eigenvalue shrinkage'': the process of modifying $S_n$'s eigenvalues but keeping its eigen-space decomposition.  
This could be as simple as the ``diagonal loading'' estimators of \cite{ledoit2004well,fuhrmann1988existence,tikhonov1943stability}, which are of the form $\alpha(x_1, \dots x_n)S_n$ plus the product  of $\beta(x_1, \dots x_n)$ and the $p_n\times p_n$ identity matrix $I_{p_n}$ for some scalar-valued functions $\alpha, \beta > 0$, but could also involve modifying the eigenvalues in a much more complicated, nonlinear way, as in \cite{donoho2018optimal} and \cite{ledoit2020analytical}.  Shrinkage estimators have been used in Tikhonov regression \cite{tikhonov1943stability} and in signal processing \cite{fuhrmann1988existence, robinson2019optimal,abrahamsson2007enhanced,chen2010shrinkage,chen2011robust},  mathematical finance \cite{ledoit2004well, ledoit2020analytical}, and many other fields \cite{zhang2009robust,pyeon2007fundamental,guo2012bayesian,ribes2013application,pirkl2011reverberation,endelman2012shrinkage,elsheikh2013iterative,bachega2011evaluating,korniotis2008habit,lotte2009efficient,markon2010modeling}.

Many authors working on high-dimensional statistical problems have considered the so-called \emph{spiked model} \cite{johnstone2001distribution}.  This is the assumption that the population covariance's eigenvalues are all equal except for finitely many ``spikes'' that are larger.  The advantage of this model is that properties of $S_n$, such as its eigenvalue- and eigenvector biases have simple asymptotic forms that hold almost surely \cite{paul2007asymptotics}---forms which are useful in areas such as principle component analysis.  
Recently, several authors have applied these forms to construct shrinkage estimators that are in a sense asymptotically optimal \cite{donoho2018optimal,robinson2020space}.  
But the spiked assumption does not always hold in detection-theoretic applications \cite{sen2015low}, in which case these estimators are less relevant. 

In this paper, we consider the case where the spiked assumption may not hold.  In this case, Ledoit and Wolf describe a shrinkage estimator they call $\tilde{S}_n$ which 
is shown to asymptotically minimize a detection-theoretically relevant loss function called MV loss among a reasonable class of estimators  \cite[Corollary~4.1]{ledoit2020analytical}.
 Our main accomplishments are twofold.  First, we show, under fairly general assumptions, that Ledoit and Wolf's estimator is asymptotically optimal among shrinkage estimators when plugged into the AMF.  Second, we give consistent estimates (Theorem~\ref{thm:main}(i) and (ii)) for the resulting AMF, enabling the practitioner to better understand its performance.  Such results exist in the spiked case \cite[Theorems~3 and 4]{robinson2020space}, but as far as we know, no one has yet successfully extended them to the nonspiked case until now.

\section{Assumptions and Main Result} \label{sec:ass}

Throughout this paper $\mathbf{1}$ will denote the indicator function, and we make the following assumptions: $X_{n}=R_{n}^{1/2}W_{n}$,
where
\begin{itemize}
\item [(H1)]The components of $W_{n}$ are i.i.d. real (or complex) random variables with zero
mean, unit variance, and have $16^{\text{th}}$ absolute central moment
bounded by a constant $C$ independent of $n$ and $p_n$;
\item[(H2)] The population covariance matrix $R_{n}$ is a nonrandom  $p_n\times p_n$
 Hermitian positive definite matrix independent of $W_{n}$;
\item [(H3)] $p_n/n\to\gamma \in (0,1)\cup (1,\infty)$ as $n\to\infty$;
\item[(H4)] $\tau_{n1} \le \tau_{n2} \le \dots \le \tau_{np_n}$ is a system of eigenvalues of $R_n$, and 
the empirical spectral distribution function (e.s.d.)
of the population covariance given by $H_{n}(\tau)=p_n^{-1}\sum_{j=1}^{p}\mathbf{1}_{[\tau_{nj},\infty)}(\tau)$
converges a.s. to a nonrandom limit $H(\tau)$ at all points of continuity of $H$;
\item[(H5)] The support of $H$ is a finite union of compact intervals and there exists a compact interval $[\underline{T},\overline{T}] \subset (0,\infty)$ such that the eigenvalues $\{\tau_{nj}\}_{j=1}^{p_n}$ are all contained in this interval for large enough $n$.  Further, we assume either $T_0=0$ or $T_0$ is some known positive number for which $T_0 \le \underline{T}$.
%
\end{itemize}
Assumptions (H1)-(H5) come from \cite{ledoit2011eigenvectors, ledoit2020analytical}.  The $16^\text{th}$-moment assumption is likely excessive: Ledoit and Wolf have performed simulations that indicate that a $4^\text{th}$-moment assumption suffices for much of what follows \cite{ledoit2020analytical}.
Assumption (H3), introduced earlier, implies that $p_n\to\infty$ whenever $n\to\infty$, and so whenever we write $n\to\infty$ it will be implied that $p_n \to\infty$ as well. The term ``asymptotically''
will always refer to the limit as $n\to\infty$ (and $p_{n}\to\infty$).

Formally, we define a shrinkage estimator as follows.  Throughout this paper $X_n = [x_{1},x_{2},\dots x_{n}]$ will be the $p_n \times n$ matrix formed from $n$ i.i.d. columns with mean 0 and covariance $R_n$,  $S_n $ will be the $p_n \times p_n$ sample covariance matrix, $\Lambda_n$ is the matrix of non-decreasing eigenvalues of $S_n$, and $U_n$ will be a random orthogonal (unitary) matrix such that $S_n = U_n \Lambda_n U_n'$.  (If $n<p_n$, for example, then $U_n$ is non-unique.)
\begin{definition} 
We say that
 $\rh_n:\bbF^{p_n\times n} \to \bbF^{p_n\times p_n}$, taking values in the cone of $p_n\times p_n$ Hermitian positive-definite matrices, is a \emph{shrinkage estimator} if and only if $\rh_n = U_n \hat{D}_n U_n'$, where $U_n$ is a (possibly random) matrix of column eigenvectors of $S_n$.
\end{definition} 
\noindent For ease of analysis, we will consider only a certain type of shrinkage estimators, described by the following definition.
\begin{definition}  \label{def:shrinkage}
Let $\rh_n = U_n \hat{D}_n U_n'$ be a shrinkage estimator.  Then $\rh_n$ is \emph{asymptotically admissible} if
\begin{enumerate}
\item[(i)] $\hat{D}_n$ is a random diagonal matrix such that the random variables $\limsup_n \left\Vert \hat{D}_n \right\Vert$ and $\limsup_n \left\Vert \hat{D}_n^{-1}\right\Vert$ are almost surely bounded, where $\left\Vert \, \cdot\, \right\Vert$ denotes the spectral norm,  and
\item[(ii)] (\emph{Limiting shrinkage function}) Letting $\toprob$ denote convergence in probability as $n,p_n\to\infty$, there exists a nonrandom continuous function $\hat{\delta}: [0,\infty) \to (0,\infty)$ such that 
\begin{equation} \label{eq:limitingdelta}
p_n^{-1} \left\Vert \rh_n - U_n \hat{\delta}(\Lambda_n) U_n' \right\Vert_{\text{Fro}}^2 \toprob 0.
\end{equation}
\end{enumerate}
\end{definition}
\noindent The continuity mentioned in (ii) is desirable since it ensures that small perturbations to eigenvalues (such as numerical errors) do not overly affect the estimator.
We also wish to consider continuous shrinkage estimators because they slightly generalize the class of shrinkage estimators satisfying \cite[Assumption~4]{ledoit2020analytical}, and newly include such famous estimators such as the Fast Maximum Likelihood (FML) estimator of \cite{steiner2000fast}  and Anderson's estimator \cite{anderson1963asymptotic}.  
We note that by (i) and (ii),  we may without loss of generality assume $\hat{\delta}$ is bounded above and below by finite nonzero numbers, and in what follows we will always do so.

Back to the context of signal detection, the detectors we consider in this paper are threshold tests on generalized AMFs: more precisely, on filters of the form $|T(\mu_n, \rh_n, y_n)|^2$, where $\rh_n$ is allowed to be any shrinkage estimator, as opposed to just the sample covariance matrix. 
The performance of such detectors depends in general on $\mu_{n}$,
and may be quite poor for some values of $\mu_n$. As a result, we
assume $\mu_{n}$ is distributed uniformly at random on the unit sphere and
will seek to make statements about the such detectors' performance conditioned on $\mu_n$ that
hold with high probability over the sphere. We encode
this information as another assumption
\begin{itemize}
\item [(H6)]$\mu_{n}$ is uniformly distributed on the unit sphere in $\bbF^{p_{n}}$
and is independent of $X_n$.
\end{itemize}


For our main result we need a couple of definitions regarding asymptotic comparisons of random variables.
\begin{definition} \label{def:asymptotic}
Suppose $a_n$ and $b_n$ are real random variables.  Then
\begin{itemize}
\item[(i)] We say $a_n$ is  \emph{i.p. asymptotically greater than or equal to} $b_n$, written $a_n \gtrsim b_n$ (i.p.), if $\max\{b_n - a_n, 0\}$ converges in probability to zero.  We say $a_n \lesssim b_n$ (i.p.) iff $-a_n \gtrsim -b_n$ (i.p.). We make a similar definition for almost sure convergence.
\item[(ii)] We say that $a_n$ is \emph{i.p. asymptotically equivalent to} to $b_n$, written $a_n \sim b_n$ (i.p.), if both $a_n \gtrsim b_n$ and $b_n \gtrsim a_n$ (i.p.) 
\end{itemize}
\end{definition}

This paper is mainly concerned with detection-theoretic properties of a particular shrinkage estimator devised by \cite{ledoit2020analytical} and referred to as $\rh_n$ from here on.  We reproduce the definition here for convenience.
\begin{definition} \label{def:lw-est}
Let $\lambda_{n1}\le \lambda_{n2}\le \dots \le \lambda_{np_n}$ be the eigenvalues of $\bS_n$, and let $\boldsymbol{\lambda}_n = (\lambda_{n1}, \dots \lambda_{np_n})$.  Let $U_n$ be a random orthogonal matrix such that $S_n = U_n \mathrm{diag}(\boldsymbol{\lambda}_n) U_n'$.  
With $[y]^+$ defined to be $\max\{y, 0\}$, $h_n$ defined to be $n^{-1/3}$, and $h_{nj}$ defined to be $\lambda_{nj} h_n$, let $a(\lambda,\boldsymbol{\lambda}_n)$ be defined by 
\begin{align*}
& \sum_{j=[p-n]^+ +1}^{p}\left\{ -\frac{3(\lambda_{ni} - \lambda_{nj})}{10\pi h_{nj}^2} + \frac{3}{4\sqrt{5}\pi h_{nj}} \right. \\
& \times \left[1 - \frac{1}{5}\left( \frac{\lambda_{ni}-\lambda_{nj}}{h_{nj}} \right)^2 \right] \left. \log\left| \frac{\sqrt{5} h_{nj} - \lambda_{ni} + \lambda_{nj}}{\sqrt{5} h_{nj} + \lambda_{ni} - \lambda_{nj}} \right| \right\}
\end{align*}
and $b(\lambda, \boldsymbol{\lambda}_n)$ be defined by
\[
\sum_{j=[p-n]^+ +1}^{p} \frac{3}{4\sqrt{5}h_{nj}} \left[ 1- \frac{1}{5}\left( \frac{\lambda_{ni} - \lambda_{nj}}{h_{nj}} \right)^2\right]^+,
\]
where the summands are defined to be zero when $j$ is not positive.  With $$\zeta_{nj}= \pi \min\{n,p_n\}^{-1}(a(\lambda_{nj},\boldsymbol{\lambda}_n)+\sqrt{-1}b(\lambda_{nj},\boldsymbol{\lambda}_n)),$$ the shrunken eigenvalues $\td_{nj}$ defined in \cite{ledoit2020analytical} are
\begin{equation} \label{eq:shrinkage-formula}
\td_{nj} \coloneqq \begin{cases}
\frac{\lambda_{nj}}{|1-p_n/n-p_n/n\lambda_{nj} \zeta_{nj} |^2}, & \text{if $\lambda_{nj} > 0 $} \\
\frac{1}{\pi (p_n/n-1)a(0,\boldsymbol{\lambda}_n)/n}, & \text{if $\lambda_{nj} = 0$}
\end{cases}
\end{equation}
Define $\hat{\delta}_{nj}$ by
\[
\begin{cases}
\lambda_{np_n}/(1+\sqrt{\gamma})^2, & \text{if $\tilde{d}_{nj}> \lambda_{np_n}/(1+\sqrt{\gamma})^2$} \\
T_0, & \text{if $\tilde{d}_{nj} < T_0$} \\
\tilde{d}_{nj}, & \text{else},
\end{cases}
\]
where $T_0$ is defined in (H5) and $\gamma=\lim p_n/n$.
Then the \emph{Ledoit-Wolf nonlinear estimator} $\rh_n$ is defined as $(U_n, \mathrm{diag}(\hat{\delta}_{n1},\dots, \hat{\delta}_{np_n}))$
\end{definition}
\noindent Code implementing the above estimator, modulo a few modifications, can be found in the supplement to \cite{ledoit2020analytical}.

Our main theorem, below, regards optimal shrinkage estimation for signal detection.


\begin{thm} \label{thm:main}
 Suppose (H1)-(H6) hold. Let $Z$ be an $\mathbb{F}$-valued random variable with standard real (complex) normal distribution.  Let $\alpha\in (0,1)$. Then if $\rh_n$ is the Ledoit-Wolf nonlinear estimator, the following hold as $n\to\infty$, $p_n\to\infty$, and $p_n/n\to \gamma$.
\begin{itemize}
    \item[(i)] (False-alarm rate) $p_0^n(\rh_n, t)$ converges in probability to $\mypr[|Z|^2 > t]$.
    \item[(ii)] (Detection rate) $p_1^n(\rh_n, t)$ is i.p. asymptotically equivalent to the random variable $$\mypr\left[\left|Z+a(\mu_n'\hat{R}_n^{-1}\mu_n)^{1/2} \right|^2 > t \mid \mu_n, X_n\right].$$
    \item[(iii)]  (Optimality) If $\rt_n$ be any asymptotically admissible shrinkage estimator, and $t_n$ and $t$ are real numbers satisfying the asymptotic constraint $p^n_0(\rt_n, t_n) \lesssim p^n_0(\rh_n, t)$ (i.p.), we have $ p^n_1(\rt_n, t_n) \lesssim p^n_1(\rh_n, t)$ (i.p.). 
\end{itemize}
\end{thm}

\begin{rem}
The  theorem above can be easily extended to the case of real scalars.  The complex case is relevant to radar beamforming, whereas the real case is relevant to sonar beamforming and hyperspectral signature detection.
\end{rem}

\section{Proofs} \label{sec:proofs}

Throughout this section and the rest of the paper, we assume (H1)-(H6), that $\rh_n$ is the Ledoit-Wolf nonlinear estimator, and that $\rt_n$ is an asymptotically admissible shrinkage estimator.  All convergence will be convergence as $n\to\infty$ and $p_n\to\infty$ unless specified otherwise. 

Our first proposition provides an approximation that will be useful throughout this paper.  First we need a lemma.

\begin{lem} \label{lem:trace-approx}
Let $\bfy$ be a random column vector that is uniformly distributed on the unit sphere $S^{p-1}$ in $\bbF^p$.  Let also $\bfA$ be a real $p\times p$ matrix.  Then there exists a constant $c > 0$ independent of $p$ and $\bfA$ such that for all $\epsilon > 0$ we have
\[
\mypr\left[ \left|\bfy' \bfA \bfy - \frac{1}{p}\tr \bfA \right| \ge \epsilon \right] \le \exp\left(-cp\epsilon^2/\left\Vert \bfA \right\Vert^2 \right).
\]
\end{lem}
\begin{proof}
By \cite[Theorem~5.1.4]{vershynin2018high} if $f$ is Lipschitz on the $\sqrt{p}$-sphere ($\sqrt{p}S^{p-1}$) and $t >0$, there exists $c' >0$ independent of $f$ and $t$ such that
\[
\mypr\left[ \left| f(\bu) - \mathbb{E}f(\bu) \right| \ge t \right] \le \exp(-c' t^2/L_f^2),
\]
where $u$ is uniformly distributed on $\sqrt{p}S^{p-1}$ and $L_f$ is a Lipschitz constant of $f$ on $\sqrt{p}S^{p-1}$.  By the rotation-invariance of the distribution of $u$, we have that $\mathbb{E}[\bu\bu'] = I_p$, the $p\times p$ identity matrix.
Thus, $$\mathbb{E} \bu'\bfA \bu = \mathbb{E} \tr(\bfA\bu\bu') = \tr\left(\bfA \mathbb{E}[\bu\bu'] \right)= \tr \bfA.$$
Further, since the gradient of $f(u)=u'\bfA u$ is $2\bfA \bu$, a Lipschitz constant of $f$ is easily seen to be $2\left\Vert \bfA\right\Vert \left\Vert \bu \right\Vert$, which is equal to $2\sqrt{p} \left\Vert \bfA\right\Vert $ on the sphere in question.  Taking $\bfy = \bu/\sqrt{p}$, then, we get
\begin{align*}
& \mypr\left[ \left|\bfy' \bfA \bfy - \frac{1}{p}\tr \bfA \right| \ge \epsilon \right]  \\
& = \mypr\left[ \left|\bu' \bfA \bu - \tr \bfA \right| \ge p \epsilon \right] \\
& = \mypr\left[\left| f(\bu) - \mathbb{E}f(\bu) \right| \ge p\epsilon \right] \\
& \le \exp\left(-c'(p\epsilon)^2/(2\sqrt{p}\left\Vert\bfA\right\Vert)^2\right).
\end{align*}
The result follows by taking $c= c'/4$.
\end{proof}

\begin{prop} \label{prop:as-trace-conv}
Let $A_n \in \bbF^{p_n \times p_n}$ be a sequence of independent Hermitian positive-definite random matrices such that the random variable given by the $\limsup$ of $\left\Vert A_n \right\Vert$ is almost surely bounded. 
Furthermore, assume that the sequences $A_n$ and $\mu_n$ are independent.
Then
\[
\left \vert \bs_n' A_n \bs_n - \frac{1}{p_n} \tr(A_n) \right \vert \toas 0.
\]
\end{prop}
\begin{proof}
By the Borel-Cantelli lemma, it suffices to prove that
\begin{equation} \label{eq:trace-approx-00}
 \mypr\left[\left|\mu'_{n}A_{n}\mu_{n}-\frac{1}{p_{n}}\tr(A_{n})\right|\ge\epsilon\right] 
\end{equation}
is summable for every $\epsilon > 0$.   
By the definition of conditional probability, the above probability is equal to
\begin{align}
& \mathbb{E} \left ( \mypr\left[\left. \left|\bs'_{n}A_{n}\bs_{n}-\frac{1}{p_{n}}\tr(A_{n})\right|\ge\epsilon\ \right| A_{n}\right] \right ) \nonumber \\
& \le \mathbb{E} \left( \exp(-cp_n\epsilon^2/\left\Vert A_n \right\Vert^2) \right). \label{eq:trace-approx-01}
\end{align}
Let $\Omega_n = \{\left\Vert A_n \right\Vert > C\}$ where $C$ is greater than the essential sup of the $\limsup$ of $\left\Vert A_n \right\Vert$.  Then \eqref{eq:trace-approx-01} is at most
\[
\mypr(\Omega_n) + \exp(-cp_n\epsilon^2/C^2).
\]
Since the second term above is summable, it suffices to show that $\mypr(\Omega_n)$ is summable.  

For this, we first note that $\mypr(\limsup_n \Omega_n) = 0$ since the $\limsup$ of $\left\Vert A_n \right\Vert$ is almost surely less than or equal to $C$.  The second Borel-Cantelli lemma (see \cite[Proposition~10.10.b]{folland1999real}), together with the fact that the $\Omega_n$'s are independent, then implies that $\mypr(\Omega_n)$ is summable, as desired.

\end{proof}

By the assumption of Definition~\ref{def:shrinkage}(i) and by (H5) and (H6) we then have the following simple proposition.
\begin{prop} \label{prop:basic-as-limits} The following two limits hold:
\begin{itemize}
\item[(i)] $\mu_n ' \rt_n^{-1} \mu_n - p_n^{-1} \tr(\rt_n^{-1}) \toas 0$
\item[(ii)] $\mu_n' \rt_n^{-1} R_n \rt_n^{-1} \mu_n -  p_n^{-1} \tr(\rt_n^{-2} R_n) \toas 0.$
\end{itemize}
\end{prop}
\noindent In the last equivalence we have used the cyclic-permutation property of trace.  

Using uniform continuity of $f(x,y)=x/y$ for $y$ outside a neighborhood of zero, together with Definition~\ref{def:shrinkage}(i), (H5), and Proposition~\ref{prop:basic-as-limits}, we get:
\begin{prop} \label{prop:xi-nu-approx}The following two limits hold:
\begin{itemize}
\item[(i)] With $\xi_n(\rt_n) := \frac{\mu'_n \rt_n^{-1} R_n \rt_n^{-1} \mu_n}{\mu_n' \rt_n^{-1} \mu_n}$, we have $$\xi_n(\rt_n) - \frac{\tr(\rt_n^{-2}R_n)}{\tr(\rt_n^{-1})} \toas 0$$
\item[(ii)] With $\nu_n(\rt_n) := \frac{\mu'_n \rt_n^{-1} \mu_n}{(\mu'_n \rt_n^{-1} R_n \rt_n^{-1} \mu_n)^{1/2}} $, we have $$ \nu_n(\rt_n) - \frac{\tr(\rt_n^{-1})}{\sqrt{p}\tr(\rt_n^{-2}R_n)^{1/2}} \toas 0.$$
\end{itemize}
\end{prop}


When $y_n$ is Gaussian, the conditional detection probability $p_1^n(\rt_n, t)$ is monotonic in $\nu_n(\rt_n)$ \cite{reed1974rapid}.  This motivates the question of how large $\nu_n(\rt_n)$ can be.  

To answer this question, we must use some results from random matrix theory.  Throughout the rest of this document we will assume that $\lambda_{n1}, \lambda_{n2}, \dots, \lambda_{np_n}$ are the eigenvalues of $S_n$ in increasing order.  Let $F_n(x) = p_n^{-1}\sum_{j=1}^{p_n} \mathbf{1}_{[\lambda_{nj},\infty)}(x)$. The Mar\v{c}enko-Pastur theorem \cite{marvcenko1967distribution,silverstein1995strong} states that under Assumptions (H1)-(H5) of this paper, there exists a Borel measure $F$ such that, almost surely, $F_n$ converges weakly to $F$.  Further, $E:=\mathrm{supp}(F)$  is a disjoint union of finitely many compact intervals \cite[Theorem~1.1]{bai1998no}.
Almost sure weak convergence means $\int h\, dF_n \toas \int h\, dF$ for any function $h$ that is bounded and continuous on $E$, which by the portmanteau theorem \cite[Theorem~2.1]{billingsley2013convergence} means that $F_n(I) \toas F(I)$ for every interval with $F(\partial I)=0$.  The Stieltijes transform of a Borel measure $G$ is given for $z$ in the upper half plane $\cplus$ as
\[
m_G(z) = \int_{-\infty}^\infty \frac{dG(t)}{t-z}.
\]
It is well known that $m_G$ is analytic in $\cplus$.  It is also well known that a sequence of Borel measures $G_n$ converges weakly to $G$ if and only if $m_{G_n}(z)$ converges pointwise to $m_G(z)$ for every $z\in\mathbb{C}^+$, relating the latter notion to the relatively simple notion of pointwise convergence of analytic functions.  Further, if $F$ is as above and $\lambda \in \bbF\backslash \{0\}$ the limit $\um_F(\lambda) = \lim_{z \to \lambda} m_F(z)$ exists and is continuous \cite[Theorem~1.1]{silverstein1995analysis}.  The following is an elaboration of the statement and proof of \cite[Theorem~3.1]{ledoit2020analytical}.

\begin{prop}[An upper limit for continuous shrinkage estimators] \label{prop:optimal}
Assume (H1)-(H6) hold and in particular that $\gamma = \lim p_n/n$. Suppose $\tilde{R}_n = U_n \tilde{D}_n U_n'$ is a asymptotically admissible shrinkage estimator. Then
\begin{equation} \label{eq:opt-ineq}
\nu_n(\tilde{R}_n) \lesssim \left(\int (1/\delta)\, dF\right)^{1/2},\qquad \mathrm{(i.p.)}
\end{equation}
where $\delta$ is the extended real-valued function given by
\[
\delta(\lambda) = 
\begin{cases}
\frac{\lambda}{|1-\gamma-\gamma\lambda \um_F (\lambda)|^2}, & \textrm{if $\lambda > 0$} \\
\frac{1}{(\gamma-1)\um_F(0)}, & \text{if $\lambda=0$}.
\end{cases}
\]
\end{prop}
\begin{proof}
First, $\delta$ is bounded on $E^o$---the interior of $E$. To see this, suppose for a contradiction there is $\epsilon >0$, $\lambda \in E^o$ such that $\delta(\lambda) \ge \overline{T}+\epsilon$, where $\overline{T}$ is the bound specified in (H5).  By continuity of $\delta$ on $E$, there is a closed interval $I$ not containing zero such that every element of $\delta(I)$ exceeds $\overline{T}+\epsilon/2$,.  Let $\dnjstar := u_{nj}' R_n u_{nj}$, where $u_{nj}$ is the $j^\text{th}$ column vector of $U_n$. 
We have
\begin{align*}
 \overline{T} F_n(I)
 & \ge p_n^{-1} \sum_{j:\, \lambda_{nj} \in I}\dnjstar \\
 & \toas \int_I \delta\, dF \\
& \ge (\overline{T}+\epsilon/2)F(I),
\end{align*}
where the convergence follows from \cite[Theorem~1.4]{ledoit2011eigenvectors}. Taking the limit implies $\lambda\notin E$, a contradiction of the initial assumption that $\delta(\lambda)\ge \overline{T}+\epsilon$. A similar argument applies to $\lambda=0$, the boundary of $E$, and obtaining the lower bound of $\underline{T}$ for all $\lambda \ge 0$.

Assume $\dt:[0,\infty)\to(0,\infty)$ is a continuous function such that
\begin{equation} \label{eq:dt-dt-square}
p_n^{-1} \sum_{j=1}^{p_n} ((\tilde{D}_n)_{jj} - \dt(\lambda_{nj}))^2 \toprob 0,
\end{equation}
as in Definition~\ref{def:shrinkage}(ii).  (Recall that we assume $\dt$ is bounded above and below by finite nonzero constants.) It follows from Cauchy-Schwarz and \eqref{eq:dt-dt-square} that 
\begin{equation} \label{eq:abs-consistency}
p_n^{-1} \sum_{j=1}^{p_n} |(\tilde{D}_n)_{jj} - \dt(\lambda_{nj})| \toprob 0.
\end{equation}
By Proposition~\ref{prop:xi-nu-approx}(ii) the numerator of $\nu_n(\rt_n)$ is a.s. asymptotically equal to to
\begin{equation} \label{eq:trtoint1}
p_n^{-1} \sum_{j=1} (\tilde{D}_n)_{jj} ^{-1}.
\end{equation}
 Using the cyclic permutation property of trace, the square of the denominator of $\nu_n(\rt_n)$ is a.s. asymptotically equal to 
\begin{equation} \label{eq:trtoint2}
 p_n^{-1} \sum_{j=1}^{p_n}(\tilde{D}_n)_{jj} ^{-2} \dnjstar ,
 \end{equation}
 where $\dnjstar$ is as before.
 By \eqref{eq:abs-consistency}, \eqref{eq:trtoint1} is i.p. asymptotically equal to
 \begin{equation} \label{eq:trtoint4}
 p_n^{-1} \sum_{j=1}^{p_n} \dt(\lambda_{nj})^{-1}
 \end{equation}
 and \eqref{eq:trtoint2} is i.p. asymptotically equal to
 \begin{equation} \label{eq:trtoint3}
 p_n^{-1} \sum_{j=1}^{p_n} \dt(\lambda_{nj})^{-2}\dnjstar.
 \end{equation}
 Let $G_n$ be the random Borel measure given by $G_n(E)=p_n^{-1}\sum_{j: \lambda_{nj}\in E}\dnjstar$.  By \cite[Theorem 1.4]{ledoit2011eigenvectors}, $G_n$ converges weakly almost surely to the Borel measure $G$ given by $dG = \delta\, dF$.  
 If $h:\mathbb{R}\to\mathbb{R}$ is bounded and continuous on $\mathbb{R}$ and $U$ is a continuity set of $G$, a generalization of the portmanteau theorem says that
 \begin{equation} \label{eq:gen-portmanteau}
 \int_U h\, dG_n \toas \int_U h\, dG.
 \end{equation}
 Using $h=\dt^{-2}$ and $U=\mathbb{R}$ gives that \eqref{eq:trtoint3} converges a.s. to
 \begin{equation} \label{eq:int1df}
 \int \dt^{-2} \delta\, dF.
 \end{equation}
In addition another application of portmanteau yields that \eqref{eq:trtoint4} converges a.s. to
\begin{equation*} \label{eq:int2df}
\int (1/\dt)\, dF.
\end{equation*}

Now, applying Cauchy-Schwarz to $\int \left((1/\dt) \delta^{1/2}\right) \delta^{-1/2}\, dF$ yields
\[
\left(\int (1/\dt) \, dF\right)^2 \le \int \dt^{-2} \delta\, dF \int (1/\delta)\, dF.
\]
Putting it all together
\begin{align*}
\nu_n(\rt_n) & \sim \frac{\int (1/\dt)\, dF}{\left( \int \dt^{-2} \delta\, dF \right)^{1/2}} \qquad \text{(i.p.)} \\
& \le \left(\int (1/\delta)\, dF \right)^{1/2},
\end{align*}
with equality if and only if $\dt$ is a constant times $\delta$ on $E$.
\end{proof}

Motivated by the above, we make the following definition.
\begin{definition} \label{def:oracle-consistency}
If $\rt_n=(U_n, \tilde{D}_n)$ is a shrinkage estimator with $$p_n^{-1}\left\Vert \tilde{D}_n - c\delta(\Lambda_n)\right\Vert_{\text{Fro}}^2 \toprob 0$$ for some $c>0$ then we say $\rt_n$ is \emph{asymptotically shrinkage-optimal}. In other words, $\rt_n$ is asymptotically shrinkage-optimal if and only if its limiting shrinkage function is $c\delta$.
If $c=1$, we say that $\rt_n$ is \emph{normalized}.  
\end{definition}

Our next result describes one of the main properties of the Ledoit-Wolf nonlinear estimator $\rh_n$.

\begin{prop} \label{thm:lw-est}
$\rh_n$ is a normalized asymptotically shrinkage-optimal estimator.
\end{prop}



\begin{proof}
We assume $\gamma > 1$.  (The case where $\gamma < 1$ is similar.)
Showing that $\rh_n$ satisfies Definitions~\ref{def:shrinkage}(i) and (ii) is trivial.  We thus seek a continuous function $\tilde{\delta}:[0,\infty)\to(0,\infty)$ that satisfies Definition~\ref{def:shrinkage}(iii) and agrees with $\delta$ on $E$, the support of $F$.  It turns out any continuous bounded function that is equal to $\delta$ on $E$ will do.  To show this, we split $p_n^{-1}\sum_{j=1}^{p_n} (\hat{\delta}_{nj} - \tilde{\delta}(\lambda_{nj}))^2$ into
\begin{equation} \label{eq:shr-opt-1}
p_n^{-1} \sum_{\substack{j:\\ \lambda_{nj} \in E^c}} (\hat{\delta}_{nj} - \tilde{\delta}(\lambda_{nj}))^2
\end{equation}
and
\begin{equation} \label{eq:shr-opt-2}
p_n^{-1} \sum_{\substack{j:\\ \lambda_{nj} \in E}} (\hat{\delta}_{nj} - \tilde{\delta}(\lambda_{nj}))^2
\end{equation}

Consider \eqref{eq:shr-opt-1}. Let $\epsilon = (1/2) \underline{T} (1-\sqrt{1/\gamma})^2$.  We claim that $\epsilon$ is an almost sure lower bound for the liminf of the smallest nonzero sample eigenvalue.  The inverse of this eigenvalue is $\left\Vert (W_n' R_n W_n)^{-1}\right\Vert$, which is bounded above by $\underline{T}^{-1}\left\Vert (W_n' W_n)^{-1}\right\Vert$.  By \cite{bai2008limit}, the latter has almost sure limit of $2\epsilon$.  Equation \eqref{eq:shr-opt-1}  then becomes
\begin{align*}
 \lesssim C F_n(E^c\cap (\epsilon, \infty)),
\end{align*}
for some constant $C$.  But this converges to zero since $E^c\cap (\epsilon, \infty)$ is a continuity set of $F$.

For \eqref{eq:shr-opt-2}, we use the result \cite[Theorem~4.1]{ledoit2020analytical}, which  states that 
\begin{equation} \label{eq:lwunif}
\sup_{j:\, \lambda_{nj}\in E} |\tilde{d}_{nj} - \delta(\lambda_{nj}) |\toprob 0.
\end{equation} Since $\delta$ is bounded above and below on $E$ by the spectral upper and lower bounds $\left\Vert R_n \right\Vert$ and $\underline{T}$ (as discussed in Proposition \ref{prop:optimal}), and since $\left\Vert R_n \right\Vert \lesssim \lambda_{np_n}/(1+\sqrt{\gamma})^2$ almost surely (see \cite{bai2008limit}), \eqref{eq:lwunif} holds equally well if the $\tilde{d}_{nj}$'s are replaced by the $\hat{\delta}_{nj}$'s:  
\begin{equation} \label{eq:lwunifimproved}
\sup_{j:\, \lambda_{nj}\in E} |\hat{\delta}_{nj} - \delta(\lambda_{nj}) |\toprob 0.
\end{equation} 
The result follows.

\end{proof}

Before we prove our main theorem, we need a couple of final lemmas.



\begin{lem} \label{lem:limit-of-xi}
Suppose $\rt_n$ is a normalized asymptotically shrinkage-optimal estimator.  Then
\begin{itemize}
\item[(i)] $	\xi_n(\rt_n) \toprob 1$
\item[(ii)] $\nu_n(\rt_n)$ is i.p. asymptotically equivalent to $\left(\mu_n' \rt_n^{-1}\mu_n\right)^{1/2}$ 
\end{itemize}
\end{lem}
\begin{proof}
We consider here the case where $\gamma > 1$.  (The case where $\gamma < 1$ is similar.) Let $\epsilon$ be as in the proof of Proposition~\ref{thm:lw-est}.


We first establish that  
\begin{equation} \label{eq:long-sesq}
 \mu_n ' \rt_n^{-1}R_n\rt_n^{-1} \mu_n
\end{equation}
 converges i.p. to $\int\delta^{-1}\, dF$.  Let $\rt_n = (U_n, \tilde{D}_n)$ with limiting shrinkage function $\delta$.  By Proposition~\ref{prop:basic-as-limits}(ii), \eqref{eq:long-sesq} is a.s. asymptotically equivalent to 
\begin{equation*} \label{eq:tracer-2r}
p_n^{-1}\sum_{j=1}^{p_n} (\tilde{D}_n)_{jj}^{-2} \dnjstar,
\end{equation*}
where $\dnjstar$ is as before. By Definition~\ref{def:shrinkage}(iii), this is i.p. asymptotically equivalent to
\begin{equation*} \label{eq:tracer-2r}
p_n^{-1}\sum_{j=1}^{p_n} \delta(\lambda_{nj})^{-2} \dnjstar,
\end{equation*}
We divide the above into two sums:
\begin{equation}\label{eq:ge-eps}
p_n^{-1} \sum_{j:\,\lambda_{nj} \ge \epsilon} \delta(\lambda_{nj})^{-2}d^*_{nj}
\end{equation}
and
\begin{equation} \label{eq:less-eps}
p_n^{-1} \sum_{j:\,\lambda_{nj} < \epsilon} \delta(\lambda_{nj})^{-2}d^*_{nj}.
\end{equation}

Consider the first expression. Since $U=[\epsilon,\max E]$ is a continuity set of the measure $G$ from Proposition~\ref{prop:optimal}, \eqref{eq:gen-portmanteau} implies that \eqref{eq:ge-eps} converges almost surely to $\int_{\epsilon}^\infty \delta^{-1}\, dF$.

Consider the expression \eqref{eq:less-eps}. Since $\epsilon$ is an almost sure lower bound on the liminf of the smallest nonzero sample eigenvalue, there is, almost surely, a number $N$ such that $n > N$ implies this eigenvalue is greater than $\epsilon$.  Assuming $n$ exceeds this number $N$, we have
\begin{align*}
\sum_{j:\, \lambda_{nj} < \epsilon} \delta(\lambda_{nj})^{-2}\dnjstar
& = \delta(0)^{-2}\sum_{j:\, \lambda_{nj} < \epsilon} \dnjstar  \\
& \toas  \delta(0)^{-2}\int_{-\infty}^\epsilon \delta\, dF \\
& = \delta(0)^{-1}F(0) \\
& = \int_{-\infty}^\epsilon (1/\delta)\, dF.
\end{align*}
where the third-to-last relation follows from \cite[Theorem~1.4]{ledoit2020analytical}.

Putting the two together yields \eqref{eq:long-sesq} is i.p. asymptotically equivalent to $\int(1/\delta)\, dF$, as claimed.

Similarly, $\mu_n' \rt_n^{-1} \mu_n$ is i.p. asymptotically equivalent to $\int (1/\delta) \, dF$.  This proves (i).  For (ii), it follows from what we have shown so far that $\nu_n(\rt_n)$ is i.p. asymptotically equivalent to $\left(\int (1/\delta)\, dF\right)^{1/2}$.  Since the latter is i.p. asymptotically equivalent to $(\mu_n'\rt_n^{-1}\mu_n)^{1/2}$, we are done.
\end{proof}

The following lemma analyzes properties of the linear functional $T(\mu_n, \rh_n, \, \cdot\, )$ (without the modulus-square applied to it).

\begin{lem} \label{lem:main}
If $\rt_n$ is a normalized asymptotically shrinkage-optimal estimator and $E$ is an open disk in $\bbF$, the following hold:
\begin{itemize}
\item[(i)] With $T_n = T(\mu_n, \rt_n, y_n)$, the random variable $\mypr[T_n \in E \mid \fH_0^n, \mu_n, X_n]$ converges in probability to $\mypr[Z\in E]$, where $Z$ is a standard circularly symmetric complex normal random variable.
\item[(ii)] The random variable $\mypr[T_n \in E \mid \fH_1^n, \mu_n, X_n]$ is i.p. asymptotically equivalent to 
\[
\mypr\left[Z\in E - a(\mu_n'\rt_n^{-1}\mu_n)^{1/2} \mid \mu_n, X_n\right].
\]
\end{itemize}
\end{lem}

\begin{proof}

Consider assertion (i). We first show that $\mypr[T_n \in E \mid \fH_0^n, \mu_n, X_n]$ converges is almost surely asymptotically equivalent to 
\begin{equation} \label{eq:i-gaussian-1}
\int_{E/\xi_n(\rt_n)^{1/2}} \phi \, dz,
\end{equation}
where $\phi$ is the standard circularly symmetric complex Gaussian. 
For this, let $y_n = R_n^{1/2}z_n \sim \fH_0^n$ and let $w_n'$ be the vector $\mu_n' \rt_n^{-1} R_n^{1/2}/(\mu_n' \rt_n^{-1}\mu_n)^{1/2}$. Then
\begin{align*}
& = \mypr\left[ T_n \in E \mid \fH_0^n, \mu_n, X_n \right] \\
& = \mypr\left[ \frac{w_n' z_n}{\xi_n(\rt_n)^{1/2}} \in \frac{1}{\xi_n(\rt_n)^{1/2}}E \mid \mu_n, X_n \right] \\
& = \mypr\left[ \frac{w_n' z_n}{\left\Vert w_n \right\Vert} \in \frac{1}{\xi_n(\rt_n)^{1/2}}E \mid  \mu_n, X_n \right]. 
\end{align*}
By the Berry-Esseen theorem  \cite{berry1941accuracy, esseen1942liapunoff}, the difference of the latter from \eqref{eq:i-gaussian-1} is bounded by
\[
\frac{C_1}{\left\Vert w_n \right\Vert} \max_{1\le i\le p_n} \frac{\rho_{n,i}}{w_{n,i}^2},
\]
for some absolute constant $C_1$, where $\rho_{n,i} = \mathbb{E}[|w_{n,i}|^3 |z_{n,i}|^3]\le C |w_{n,i}|^3$.  Using Definition~\ref{def:shrinkage}(ii), it is straightforward to show that, asymptotically, the quantity $\left\Vert w_n \right\Vert$ is bounded below.  Thus, we wish to show that $\max_i |w_{n,i}|$ goes to zero rapidly in probability.

For this, observe that $|w_{n,i}|=(w_{n,i}\overline{w_{n,i}})^{1/2}$ can be written as
$$(\mu_n' \rt_n^{-1}R_n^{1/2}e_{n,i}e_{n,i}'R_n^{1/2}\rt_n^{-1}\mu_n)^{1/2} = : \epsilon_{n,i}^{1/2}.$$
Then,
\begin{align}
& \mypr[\max_i |w_{n,i}| > \epsilon]  \nonumber \\
&  = \mypr[\max_i \epsilon_{n,i} > \epsilon^2] \nonumber\\
& = \mypr[\exists i:\ \epsilon_{n,i} > \epsilon^2] \nonumber \\
& \le \sum_{i=1}^{p_n} \mypr[\epsilon_{n,i} > \epsilon^2] \nonumber \\
& \le \sum_{i=1}^{p_n} \mypr[|\epsilon_{n,i} - \tilde{\epsilon}_{n,i}| > \epsilon^2/2] +  \sum_{i=1}^{p_n} \mypr[\tilde{\epsilon}_{n,i} > \epsilon^2/2],\label{eq:main-berry-estimate} 
\end{align}
where
\begin{align*}
\tilde{\epsilon}_{n,i} & = p_n^{-1}\tr(\rt_n^{-1}R_n^{1/2}e_{n,i}e_{n,i}'R_n^{1/2}\rt_n^{-1}) \\
 & = p_n^{-1}e_{n,i}'\rt_n^{-1}R_n\rt_n^{-1}e_{n,i} \\
 & \le p_n^{-1} \left\Vert R_n \right\Vert \left\Vert \rt_n^{-1} \right\Vert^2,
\end{align*}
Thus, with $Y_n = \left\Vert \rt_n^{-1} \right\Vert^2$, the second term of \eqref{eq:main-berry-estimate} is asymptotically bounded above by
\begin{equation} \label{eq:second-berry}
p_n \mypr[ p_n^{-1} Y_n > \epsilon']
\end{equation}
for $\epsilon' = \epsilon^2/(2\underline{T})$. By Chernoff's inequality, this is  bounded above by
\[
p_n \exp(-\epsilon' p_n) \mathbb{E} \exp(Y_n).
\]
But since $\limsup_n Y_n$ is almost surely bounded (see Definition~\ref{def:shrinkage}(i)), say by $C$, we have that the above is bounded above by
\[
p_n \exp(-\epsilon' p_n) \exp(2C)
\]
for $n$ large enough.
Thus, the second term in \eqref{eq:main-berry-estimate} goes to zero exponentially fast.  Further the first term in \eqref{eq:main-berry-estimate} converges to zero exponentially fast by the proof of Proposition~\ref{prop:as-trace-conv} with $A_n$ equal to $\rt_n^{-1}R_n^{1/2} e_{n,i} e_{n,i}' R_n^{1/2} \rt_n^{-1}$.  By Borel-Cantelli and the continuous mapping theorem, this proves the almost sure asymptotic equivalence asserted. 

An upper bound for the difference between \eqref{eq:i-gaussian-1} and $\int_E \phi\, dz$ is $\left\Vert \phi \right\Vert_\infty$ times
\begin{equation} \label{eq:symm-diff}
 \lambda\left((E/\xi_n(\rt_n)^{1/2})\Delta E\right),
\end{equation}
where $\lambda$ is Lebesgue measure and $\Delta$ is the symmetric difference. Using continuity of the above inset expression in $\xi_n(\rt_n)^{1/2}$ and
Lemma~\ref{lem:limit-of-xi}, \eqref{eq:symm-diff} converges in probability to zero. The proof of part (i) is complete.

Now consider (ii). If $\rt_n$ is any shrinkage estimator, then
$p_1^n(\rt_n, t)$ is a.s. asymptotically equivalent to 
\begin{equation} \label{eq:main-thm-iii}
\int_{E/\xi_n(\rt_n)^{1/2} - a\nu_n(\rt_n)} \phi\, dz.
\end{equation}
This follows since $\fH_1^n$ data are just $\fH_0^n$ data shifted by $\mu_n$ and the fact that $$\nu_n(\rt_n)  = (\mu_n'\rt_n^{-1}\mu_n)^{1/2}/\xi_n(\rt_n)^{1/2}.$$
Applying a similar continuity argument to the one applied to \eqref{eq:symm-diff}, we get that \eqref{eq:main-thm-iii} is i.p. asymptotically equivalent to $$\int_{E-a(\mu_n'\rt_n^{-1}\mu_n)^{1/2}} \phi\, dz,$$ as desired.
\end{proof}

We now prove the main theorem. 

\begin{proof}[Proof of Theorem~\ref{thm:main}]
For (i) and (ii) simply use Lemma~\ref{lem:main} with $\rt_n=\rh_n$ and disks centered at the origin.



Next consider (iii). 
Suppose $t_n$ and $t$ and $\rt_n$ are such that $p_0(\rt_n, t_n) \lesssim p_0(\rh_n, t)$.  By the Berry-Esseen arguments above, this means
\[
\mypr\left[|Z|^2/\xi_n(\rt_n) > t_n \mid \mu_n, X_n\right] \lesssim \mypr[|Z|^2 > t] \qquad \text{(i.p.)}.
\]
This implies that $t_n\xi_n(\rt_n) \gtrsim t$.  Thus,
\begin{align*}
    & p_1^n(\rt_n, t_n) \\ 
     & \lesssim \mypr[|T(\mu_n, \rt_n, y_n)|^2 > t/\xi_n(\rt_n) \mid \fH_1^n, \mu_n, X_n] \qquad \text{(i.p.)},
\end{align*}
which is i.p. asymptotically equivalent to
\[
\mypr[|Z/\xi_n(\rt_n)^{1/2}+a\nu_n(\rt_n)|^2 > t/\xi_n(\rt_n) \mid \mu_n, X_n]. 
\]
Using the shrinkage-optimality of $\rh_n$, Lemma~\ref{lem:limit-of-xi}(i), and another dominated convergence argument, the inset expression immediately above is i.p. asymptotically less than or equal to 
\[
\mypr[|Z+a\nu_n(\rh_n)|^2 > t],
\]
which by part (ii) of this theorem is i.p. asymptotically equivalent to $
p_1^n(\rh_n, t)$.
The proof is complete.

\end{proof}

\section*{Acknowledgements}
This work was supported by the United States Air Force Sensors Directorate and AFOSR grants 19RYCOR036 and 22RYCOR007, and ARO grant W911NF-15-1-0479. However, the views and opinions expressed in this article are those of the authors and do not necessarily reflect the official policy or position of any agency of the U.S. government. Examples of analysis performed within this article are only examples. Assumptions made within the analysis are also not reflective of the position of any U.S. Government entity. The Public Affairs approval number of this document is AFRL-2021-4155.

\bibliographystyle{plain}
\bibliography{information-geometry-bib2}

\begin{thebibliography}{10}

\bibitem{abrahamsson2007enhanced}
Richard Abrahamsson, Yngve Selen, and Petre Stoica.
\newblock Enhanced covariance matrix estimators in adaptive beamforming.
\newblock In {\em 2007 IEEE International Conference on Acoustics, Speech and
  Signal Processing-ICASSP'07}, volume~2, pages II--969. IEEE, 2007.

\bibitem{anderson1963asymptotic}
Theodore~Wilbur Anderson.
\newblock Asymptotic theory for principal component analysis.
\newblock {\em The Annals of Mathematical Statistics}, 34(1):122--148, 1963.

\bibitem{bachega2011evaluating}
Leonardo~R. Bachega, James Theiler, and Charles~A. Bouman.
\newblock Evaluating and improving local hyperspectral anomaly detectors.
\newblock In {\em 2011 IEEE Applied Imagery Pattern Recognition Workshop
  (AIPR)}, pages 1--8. IEEE, 2011.

\bibitem{bai1998no}
Zhi-Dong Bai, Jack~W. Silverstein, et~al.
\newblock No eigenvalues outside the support of the limiting spectral
  distribution of large-dimensional sample covariance matrices.
\newblock {\em The Annals of Probability}, 26(1):316--345, 1998.

\bibitem{bai2008limit}
Zhi-Dong Bai and Yong-Qua Yin.
\newblock Limit of the smallest eigenvalue of a large dimensional sample
  covariance matrix.
\newblock In {\em Advances In Statistics}, pages 108--127. World Scientific,
  2008.

\bibitem{berry1941accuracy}
Andrew~C. Berry.
\newblock The accuracy of the {G}aussian approximation to the sum of
  independent variates.
\newblock {\em Transactions of the American Mathematical Society},
  49(1):122--136, 1941.

\bibitem{billingsley2013convergence}
Patrick Billingsley.
\newblock {\em Convergence of Probability Measures}.
\newblock John Wiley \& Sons, 2013.

\bibitem{chen2010shrinkage}
Yilun Chen, Ami Wiesel, Yonina~C. Eldar, and Alfred~O. Hero.
\newblock Shrinkage algorithms for {MMSE} covariance estimation.
\newblock {\em IEEE Transactions on Signal Processing}, 58(10):5016--5029,
  2010.

\bibitem{chen2011robust}
Yilun Chen, Ami Wiesel, and Alfred~O. Hero.
\newblock Robust shrinkage estimation of high-dimensional covariance matrices.
\newblock {\em IEEE Transactions on Signal Processing}, 59(9):4097--4107, 2011.

\bibitem{donoho2006most}
David~L Donoho.
\newblock For most large underdetermined systems of linear equations the
  minimal $\ell_1$-norm solution is also the sparsest solution.
\newblock {\em Communications on Pure and Applied Mathematics: A Journal Issued
  by the Courant Institute of Mathematical Sciences}, 59(6):797--829, 2006.

\bibitem{donoho2018optimal}
David~L. Donoho, Matan Gavish, and Iain~M. Johnstone.
\newblock Optimal shrinkage of eigenvalues in the spiked covariance model.
\newblock {\em Annals of Statistics}, 46(4):1742, 2018.

\bibitem{elsheikh2013iterative}
Ahmed~H. Elsheikh, Mary~F. Wheeler, and Ibrahim Hoteit.
\newblock An iterative stochastic ensemble method for parameter estimation of
  subsurface flow models.
\newblock {\em Journal of Computational Physics}, 242:696--714, 2013.

\bibitem{endelman2012shrinkage}
Jeffrey~B. Endelman and Jean-Luc Jannink.
\newblock Shrinkage estimation of the realized relationship matrix.
\newblock {\em G3: Genes, Genomes, Genetics}, 2(11):1405--1413, 2012.

\bibitem{esseen1942liapunoff}
Carl-Gustav Esseen.
\newblock On the {L}iapunoff limit of error in the theory of probability.
\newblock {\em Ark. Mat. Astr. Fys.}, 28(1--19), 1942.

\bibitem{fan2008sure}
Jianqing Fan and Jinchi Lv.
\newblock Sure independence screening for ultrahigh dimensional feature space.
\newblock {\em Journal of the Royal Statistical Society: Series B (Statistical
  Methodology)}, 70(5):849--911, 2008.

\bibitem{folland1999real}
Gerald~B. Folland.
\newblock {\em Real Analysis: Modern Techniques and Their Applications},
  volume~40.
\newblock John Wiley \& Sons, 1999.

\bibitem{fuhrmann1988existence}
Daniel~R. Fuhrmann and Michael~I. Miller.
\newblock On the existence of positive-definite maximum-likelihood estimates of
  structured covariance matrices.
\newblock {\em IEEE Transactions on Information Theory}, 34(4):722--729, 1988.

\bibitem{guo2012bayesian}
Syuan-Ming Guo, Jun He, Nilah Monnier, Guangyu Sun, Thorsten Wohland, and Mark
  Bathe.
\newblock Bayesian approach to the analysis of fluorescence correlation
  spectroscopy data ii: {A}pplication to simulated and in vitro data.
\newblock {\em Analytical chemistry}, 84(9):3880--3888, 2012.

\bibitem{hastie1995penalized}
Trevor Hastie, Andreas Buja, and Robert Tibshirani.
\newblock Penalized discriminant analysis.
\newblock {\em The Annals of Statistics}, pages 73--102, 1995.

\bibitem{johnstone2001distribution}
Iain~M. Johnstone.
\newblock On the distribution of the largest eigenvalue in principal components
  analysis.
\newblock {\em Annals of Statistics}, pages 295--327, 2001.

\bibitem{korniotis2008habit}
George~M. Korniotis.
\newblock Habit formation, incomplete markets, and the significance of regional
  risk for expected returns.
\newblock {\em The Review of Financial Studies}, 21(5):2139--2172, 2008.

\bibitem{ledoit2011eigenvectors}
Olivier Ledoit and Sandrine P{\'e}ch{\'e}.
\newblock Eigenvectors of some large sample covariance matrix ensembles.
\newblock {\em Probability Theory and Related Fields}, 151(1-2):233--264, 2011.

\bibitem{ledoit2004well}
Olivier Ledoit and Michael Wolf.
\newblock A well-conditioned estimator for large-dimensional covariance
  matrices.
\newblock {\em Journal of Multivariate Analysis}, 88(2):365--411, 2004.

\bibitem{ledoit2017nonlinear}
Olivier Ledoit and Michael Wolf.
\newblock Nonlinear shrinkage of the covariance matrix for portfolio selection:
  {M}arkowitz meets {G}oldilocks.
\newblock {\em The Review of Financial Studies}, 30(12):4349--4388, 2017.

\bibitem{ledoit2020analytical}
Olivier Ledoit and Michael Wolf.
\newblock Analytical nonlinear shrinkage of large-dimensional covariance
  matrices.
\newblock {\em Annals of Statistics}, 48(5):3043--3065, 2020.

\bibitem{lotte2009efficient}
Fabien Lotte and Cuntai Guan.
\newblock An efficient {}p300-based brain-computer interface with minimal
  calibration time.
\newblock 2009.

\bibitem{marvcenko1967distribution}
Vladimir~A. Mar{\v{c}}enko and Leonid~Andreevich Pastur.
\newblock Distribution of eigenvalues for some sets of random matrices.
\newblock {\em Mathematics of the USSR-Sbornik}, 1(4):457, 1967.

\bibitem{markon2010modeling}
K.~E. Markon.
\newblock Modeling psychopathology structure: {A} symptom-level analysis of
  {A}xis {I} and {II} disorders.
\newblock {\em Psychological medicine}, 40(2):273, 2010.

\bibitem{paul2007asymptotics}
Debashis Paul.
\newblock Asymptotics of sample eigenstructure for a large dimensional spiked
  covariance model.
\newblock {\em Statistica Sinica}, pages 1617--1642, 2007.

\bibitem{pirkl2011reverberation}
Ryan~J. Pirkl, Kate~A. Remley, and Christian S.~L{\"o}tb{\"a}ck Patan{\'e}.
\newblock Reverberation chamber measurement correlation.
\newblock {\em IEEE Transactions on Electromagnetic Compatibility},
  54(3):533--545, 2011.

\bibitem{pyeon2007fundamental}
Dohun Pyeon, Michael~A. Newton, Paul~F. Lambert, Johan~A. Den~Boon, Srikumar
  Sengupta, Carmen~J. Marsit, Craig~D. Woodworth, Joseph~P. Connor, Thomas~H.
  Haugen, and Elaine~M. Smith.
\newblock Fundamental differences in cell cycle deregulation in human
  papillomavirus--positive and human papillomavirus--negative head/neck and
  cervical cancers.
\newblock {\em Cancer research}, 67(10):4605--4619, 2007.

\bibitem{reed1974rapid}
Irving~S. Reed, John~D. Mallett, and Lawrence~E. Brennan.
\newblock Rapid convergence rate in adaptive arrays.
\newblock {\em IEEE Transactions on Aerospace and Electronic Systems},
  (6):853--863, 1974.

\bibitem{ribes2009adaptation}
Aur{\'e}lien Ribes, Jean-Marc Aza{\"\i}s, and Serge Planton.
\newblock Adaptation of the optimal fingerprint method for climate change
  detection using a well-conditioned covariance matrix estimate.
\newblock {\em Climate Dynamics}, 33(5):707--722, 2009.

\bibitem{ribes2013application}
Aur{\'e}lien Ribes, Serge Planton, and Laurent Terray.
\newblock Application of regularised optimal fingerprinting to attribution.
  {P}art i: {M}ethod, properties and idealised analysis.
\newblock {\em Climate dynamics}, 41(11-12):2817--2836, 2013.

\bibitem{robey1992cfar}
Frank~C. Robey, Daniel~R. Fuhrmann, Edward~J. Kelly, and Ramon Nitzberg.
\newblock A {CFAR} adaptive matched filter detector.
\newblock {\em IEEE Transactions on Aerospace and Electronic Systems},
  28(1):208--216, 1992.

\bibitem{robinson2019optimal}
Benjamin~D. Robinson.
\newblock Optimal rotation-equivariant covariance estimation for detection of
  high-dimensional signals.
\newblock In {\em 2019 IEEE Radar Conference (RadarConf)}, pages 1--6. IEEE,
  2019.

\bibitem{robinson2020space}
Benjamin~D. Robinson, Robert Malinas, and Alfred~O. Hero~III.
\newblock Space-time adaptive detection at low sample support.
\newblock {\em arXiv preprint arXiv:2010.03388}, 2020.

\bibitem{scharf1994matched}
Louis~L. Scharf and Benjamin Friedlander.
\newblock Matched subspace detectors.
\newblock {\em IEEE Transactions on Signal Processing}, 42(8):2146--2157, 1994.

\bibitem{sen2015low}
Satyabrata Sen.
\newblock Low-rank matrix decomposition and spatio-temporal sparse recovery for
  {STAP} radar.
\newblock {\em IEEE Journal of Selected Topics in Signal Processing},
  9(8):1510--1523, 2015.

\bibitem{silverstein1995strong}
Jack~W. Silverstein.
\newblock Strong convergence of the empirical distribution of eigenvalues of
  large dimensional random matrices.
\newblock {\em Journal of Multivariate Analysis}, 55(2):331--339, 1995.

\bibitem{silverstein1995analysis}
Jack~W. Silverstein, Sang-Il Choi, et~al.
\newblock Analysis of the limiting spectral distribution of large dimensional
  random matrices.
\newblock {\em Journal of Multivariate Analysis}, 54(2):295--309, 1995.

\bibitem{stein1975estimation}
Charles Stein.
\newblock Estimation of a covariance matrix, {R}ietz lecture.
\newblock In {\em 39th Annual Meeting IMS, Atlanta, GA, 1975}, 1975.

\bibitem{stein1986lectures}
Charles Stein.
\newblock Lectures on the theory of estimation of many parameters.
\newblock {\em Journal of Soviet Mathematics}, 34(1):1373--1403, 1986.

\bibitem{steiner2000fast}
Michael Steiner and Karl Gerlach.
\newblock Fast converging adaptive processor for a structured covariance
  matrix.
\newblock {\em IEEE Transactions on Aerospace and Electronic Systems},
  36(4):1115--1126, 2000.

\bibitem{theilera2007beyond}
James Theilera, Bernard~R. Foyb, and Andrew~M. Frasera.
\newblock Beyond the adaptive matched filter: Nonlinear detectors for weak
  signals in high-dimensional clutter.
\newblock In {\em Proc. SPIE}, volume 6565, pages 6565--02, 2007.

\bibitem{tikhonov1943stability}
Andrey~Nikolayevich Tikhonov.
\newblock On the stability of inverse problems.
\newblock In {\em Dokl. Akad. Nauk SSSR}, volume~39, pages 195--198, 1943.

\bibitem{vershynin2018high}
Roman Vershynin.
\newblock {\em High-Dimensional Probability: An Introduction with Applications
  in Data Science}, volume~47.
\newblock Cambridge University Press, 2018.

\bibitem{ward1998space}
James Ward.
\newblock Space-time adaptive processing for airborne radar.
\newblock 1998.

\bibitem{zhang2009robust}
Youwen Zhang, Dajun Sun, and Dianlun Zhang.
\newblock Robust adaptive acoustic vector sensor beamforming using automated
  diagonal loading.
\newblock {\em Applied acoustics}, 70(8):1029--1033, 2009.

\bibitem{zhao2006model}
Peng Zhao and Bin Yu.
\newblock On model selection consistency of {L}asso.
\newblock {\em The Journal of Machine Learning Research}, 7:2541--2563, 2006.

\end{thebibliography}

\ifCLASSOPTIONcaptionsoff
  \newpage
\fi

\end{document}